\documentclass[12pt]{amsart}
\usepackage{amsmath,amsthm,amsfonts,amssymb,mathrsfs}
\usepackage{color}

\usepackage{tikz}

\usepackage{amssymb,cite}
\usepackage[colorlinks,plainpages,citecolor=magenta, linkcolor=blue, backref]{hyperref}

\usepackage{hyperref}
\date{\today}

\input xymatrix
\xyoption{all}

\usepackage{hyperref}
\input{xy}
 \xyoption{all}
 \xyoption{arc}

\newtheorem{theorem}{Theorem}[section]

\newtheorem{proposition}[theorem]{Proposition}
\newtheorem{corollary}[theorem]{Corollary}

\newtheorem{lemma}[theorem]{Lemma}
\theoremstyle{definition}

\newtheorem{remark}[theorem]{Remark}

\begin{document}

\title[On injective  endomorphisms of the semigroup $\boldsymbol{B}_{\mathbb{Z}}^{\mathscr{F}}$]{On injective endomorphisms of the semigroup $\boldsymbol{B}_{\mathbb{Z}}^{\mathscr{F}^2}$ with the two-element family $\mathscr{F}^2$ of inductive nonempty subsets of $\omega$}
\author{Oleg Gutik and Inna Pozdniakova}
\address{Ivan Franko National University of Lviv, Universytetska 1, Lviv, 79000, Ukraine}
\email{oleg.gutik@lnu.edu.ua, pozdnyakova.inna@gmail.com}



\begin{abstract}
We describe injective endomorphisms of the semigroup $\boldsymbol{B}_{Z\mathbb{}}^{\mathscr{F}^2}$ with the two-element family $\mathscr{F}^2$ of inductive nonempty subsets of $\omega$. In particular we show that every injective endomorphism $\mathfrak{e}$ of $\boldsymbol{B}_{Z\mathbb{}}^{\mathscr{F}^2}$ is presented in the form $\mathfrak{e}=\mathfrak{e}_0\mathfrak{a}$, where $\mathfrak{e}_0$ is an injective  $(0,0,[0))$-endomorphism of $\boldsymbol{B}_{\mathbb{Z}}^{\mathscr{F}^2}$ and $\mathfrak{a}$ is an automorphism $\mathfrak{a}$ of $\boldsymbol{B}_{\mathbb{Z}}^{\mathscr{F}^2}$. Also we describe all injective  $(0,0,[0))$-endomorphisms $\mathfrak{e}_0$ of $\boldsymbol{B}_{\mathbb{Z}}^{\mathscr{F}^2}$, i.e., such that $(0,0,[0))\mathfrak{e}_0=(0,0,[0))$.
\end{abstract}

\maketitle


\section{Introduction, motivation and main definitions}

We shall follow the terminology of~\cite{Clifford-Preston-1961, Clifford-Preston-1967, Lawson=1998}. By $\omega$ we denote the set of all non-negative integers and by $\mathbb{Z}$ the set of all integers.

A subset $A$ of $\omega$ is said to be \emph{inductive}, if $i\in A$ implies $i+1\in A$. Obvious, that $\varnothing$ is an inductive subset of $\omega$.

\begin{remark}[\!\cite{Gutik-Mykhalenych=2021}]\label{remark-1.1}
\begin{enumerate}
  \item\label{remark-1.1(1)} By Lemma~6 from \cite{Gutik-Mykhalenych=2020} a nonempty subset $F\subseteq \omega$ is inductive in $\omega$ if and only $(-1+F)\cap F=F$.
  \item\label{remark-1.1(2)} Since the set $\omega$ with the usual order is well-ordered, for any nonempty inductive subset $F$ in $\omega$ there exists nonnegative integer $n_F\in\omega$ such that $[n_F)=F$.
  \item\label{remark-1.1(3)} Statement \eqref{remark-1.1(2)} implies that the intersection of an arbitrary finite family of nonempty inductive subsets in $\omega$ is a nonempty inductive subset of  $\omega$.
\end{enumerate}
\end{remark}

Let $\mathscr{P}(\omega)$ be  the family of all subsets of $\omega$. For any $F\in\mathscr{P}(\omega)$ and $n,m\in\omega$ we put $n-m+F=\{n-m+k\colon k\in F\}$ if $F\neq\varnothing$ and $n-m+\varnothing=\varnothing$. A subfamily $\mathscr{F}\subseteq\mathscr{P}(\omega)$ is called \emph{${\omega}$-closed} if $F_1\cap(-n+F_2)\in\mathscr{F}$ for all $n\in\omega$ and $F_1,F_2\in\mathscr{F}$. For any $a\in\omega$ we denote $[a)=\{x\in\omega\colon x\geqslant a\}$.

A semigroup $S$ is called {\it inverse} if for any
element $x\in S$ there exists a unique $x^{-1}\in S$ such that
$xx^{-1}x=x$ and $x^{-1}xx^{-1}=x^{-1}$. The element $x^{-1}$ is
called the {\it inverse of} $x\in S$. If $S$ is an inverse
semigroup, then the function $\operatorname{inv}\colon S\to S$
which assigns to every element $x$ of $S$ its inverse element
$x^{-1}$ is called the {\it inversion}.


If $S$ is a semigroup, then we shall denote the subset of all
idempotents in $S$ by $E(S)$. If $S$ is an inverse semigroup, then
$E(S)$ is closed under multiplication and we shall refer to $E(S)$ as a
\emph{band} (or the \emph{band of} $S$). Then the semigroup
operation on $S$ determines the following partial order $\preccurlyeq$
on $E(S)$: $e\preccurlyeq f$ if and only if $ef=fe=e$. This order is
called the {\em natural partial order} on $E(S)$. A \emph{semilattice} is a commutative semigroup of idempotents.

If $S$ is an inverse semigroup then the semigroup operation on $S$ determines the following partial order $\preccurlyeq$
on $S$: $s\preccurlyeq t$ if and only if there exists $e\in E(S)$ such that $s=te$. This order is
called the {\em natural partial order} on $S$ \cite{Wagner-1952}.

The \emph{bicyclic monoid} or the \emph{bicyclic semigroup} ${\mathscr{C}}(p,q)$ is the semigroup with the identity $1$ generated by two elements $p$ and $q$ subjected only to the condition $pq=1$. The semigroup operation on ${\mathscr{C}}(p,q)$ is determined as
follows:
\begin{equation*}
    q^kp^l\cdot q^mp^n=q^{k+m-\min\{l,m\}}p^{l+n-\min\{l,m\}}.
\end{equation*}
It is well known that the bicyclic monoid ${\mathscr{C}}(p,q)$ is a bisimple (and hence simple) combinatorial $E$-unitary inverse semigroup and every non-trivial congruence on ${\mathscr{C}}(p,q)$ is a group congruence \cite{Clifford-Preston-1961}.

On the set $\boldsymbol{B}_{\omega}=\omega\times\omega$ we define the semigroup operation ``$\cdot$'' in the following way
\begin{equation}\label{eq-1.1}
  (i_1,j_1)\cdot(i_2,j_2)=
  \left\{
    \begin{array}{ll}
      (i_1-j_1+i_2,j_2), & \hbox{if~} j_1\leqslant i_2;\\
      (i_1,j_1-i_2+j_2), & \hbox{if~} j_1\geqslant i_2.
    \end{array}
  \right.
\end{equation}
It is well known that the bicyclic monoid $\mathscr{C}(p,q)$ to the semigroup $\boldsymbol{B}_{\omega}$ is isomorphic by the mapping $\mathfrak{h}\colon \mathscr{C}(p,q)\to \boldsymbol{B}_{\omega}$, $q^kp^l\mapsto (k,l)$ (see: \cite[Section~1.12]{Clifford-Preston-1961} or \cite[Exercise IV.1.11$(ii)$]{Petrich-1984}).


Next we shall describe the construction which is introduced in \cite{Gutik-Mykhalenych=2020}.

Let $\boldsymbol{B}_{\omega}$ be the bicyclic monoid and $\mathscr{F}$ be an ${\omega}$-closed subfamily of $\mathscr{P}(\omega)$. On the set $\boldsymbol{B}_{\omega}\times\mathscr{F}$ we define the semigroup operation ``$\cdot$'' in the following way
\begin{equation}\label{eq-1.2}
  (i_1,j_1,F_1)\cdot(i_2,j_2,F_2)=
  \left\{
    \begin{array}{ll}
      (i_1-j_1+i_2,j_2,(j_1-i_2+F_1)\cap F_2), & \hbox{if~} j_1\leqslant i_2;\\
      (i_1,j_1-i_2+j_2,F_1\cap (i_2-j_1+F_2)), & \hbox{if~} j_1\geqslant i_2.
    \end{array}
  \right.
\end{equation}
In \cite{Gutik-Mykhalenych=2020} is proved that if the family $\mathscr{F}\subseteq\mathscr{P}(\omega)$ is ${\omega}$-closed then $(\boldsymbol{B}_{\omega}\times\mathscr{F},\cdot)$ is a semigroup. Moreover, if an ${\omega}$-closed family  $\mathscr{F}\subseteq\mathscr{P}(\omega)$ contains the empty set $\varnothing$ then the set
$ 
  \boldsymbol{I}=\{(i,j,\varnothing)\colon i,j\in\omega\}
$ 
is an ideal of the semigroup $(\boldsymbol{B}_{\omega}\times\mathscr{F},\cdot)$. For any ${\omega}$-closed family $\mathscr{F}\subseteq\mathscr{P}(\omega)$ the following semigroup
\begin{equation*}
  \boldsymbol{B}_{\omega}^{\mathscr{F}}=
\left\{
  \begin{array}{ll}
    (\boldsymbol{B}_{\omega}\times\mathscr{F},\cdot)/\boldsymbol{I}, & \hbox{if~} \varnothing\in\mathscr{F};\\
    (\boldsymbol{B}_{\omega}\times\mathscr{F},\cdot), & \hbox{if~} \varnothing\notin\mathscr{F}
  \end{array}
\right.
\end{equation*}
is defined in \cite{Gutik-Mykhalenych=2020}. The semigroup $\boldsymbol{B}_{\omega}^{\mathscr{F}}$ generalizes the bicyclic monoid and the countable semigroup of matrix units. It is proven in \cite{Gutik-Mykhalenych=2020} that $\boldsymbol{B}_{\omega}^{\mathscr{F}}$ is a combinatorial inverse semigroup and Green's relations, the natural partial order on $\boldsymbol{B}_{\omega}^{\mathscr{F}}$ and its set of idempotents are described.
Here, the criteria when the semigroup $\boldsymbol{B}_{\omega}^{\mathscr{F}}$ is simple, $0$-simple, bisimple, $0$-bisimple, or it has the identity, are given.
In particularly in \cite{Gutik-Mykhalenych=2020} it is proved that the semigroup $\boldsymbol{B}_{\omega}^{\mathscr{F}}$ is isomorphic to the semigrpoup of ${\omega}{\times}{\omega}$-matrix units if and only if $\mathscr{F}$ consists of a singleton set and the empty set, and $\boldsymbol{B}_{\omega}^{\mathscr{F}}$ is isomorphic to the bicyclic monoid if and only if $\mathscr{F}$ consists of a non-empty inductive subset of $\omega$.

Group congruences on the semigroup  $\boldsymbol{B}_{\omega}^{\mathscr{F}}$ and its homomorphic retracts  in the case when an ${\omega}$-closed family $\mathscr{F}$ consists of inductive non-empty subsets of $\omega$ are studied in \cite{Gutik-Mykhalenych=2021}. It is proven that a congruence $\mathfrak{C}$ on $\boldsymbol{B}_{\omega}^{\mathscr{F}}$ is a group congruence if and only if its restriction on a subsemigroup of $\boldsymbol{B}_{\omega}^{\mathscr{F}}$, which is isomorphic to the bicyclic semigroup, is not the identity relation. Also in \cite{Gutik-Mykhalenych=2021}, all non-trivial homomorphic retracts and isomorphisms  of the semigroup $\boldsymbol{B}_{\omega}^{\mathscr{F}}$ are described.

In \cite{Gutik-Lysetska=2021, Lysetska=2020} the algebraic structure of the semigroup $\boldsymbol{B}_{\omega}^{\mathscr{F}}$ is established in the case when ${\omega}$-closed family $\mathscr{F}$ consists of atomic subsets of ${\omega}$.

The set $\boldsymbol{B}_{\mathbb{Z}}=\mathbb{Z}\times\mathbb{Z}$ with the semigroup operation defined by formula \eqref{eq-1.1} is called the \emph{extended bicyclic semigroup} \cite{Warne-1968}. On the set $\boldsymbol{B}_{\mathbb{Z}}\times\mathscr{F}$, where $\mathscr{F}$ is an ${\omega}$-closed subfamily of $\mathscr{P}(\omega)$, we define the semigroup operation ``$\cdot$'' by formula \eqref{eq-1.2}. In \cite{Gutik-Pozdniakova=2021} it is proved that $(\boldsymbol{B}_{\mathbb{Z}}\times\mathscr{F},\cdot)$ is a semigroup. Moreover, if an ${\omega}$-closed family  $\mathscr{F}\subseteq\mathscr{P}(\omega)$ contains the empty set $\varnothing$ then the set
$ 
  \boldsymbol{I}=\{(i,j,\varnothing)\colon i,j\in\mathbb{Z}\}
$ 
is an ideal of the semigroup $(\boldsymbol{B}_{\mathbb{Z}}\times\mathscr{F},\cdot)$. For any ${\omega}$-closed family $\mathscr{F}\subseteq\mathscr{P}(\omega)$ the following semigroup
\begin{equation*}
  \boldsymbol{B}_{\mathbb{Z}}^{\mathscr{F}}=
\left\{
  \begin{array}{ll}
    (\boldsymbol{B}_{\mathbb{Z}}\times\mathscr{F},\cdot)/\boldsymbol{I}, & \hbox{if~} \varnothing\in\mathscr{F};\\
    (\boldsymbol{B}_{\mathbb{Z}}\times\mathscr{F},\cdot), & \hbox{if~} \varnothing\notin\mathscr{F}
  \end{array}
\right.
\end{equation*}
is defined in \cite{Gutik-Pozdniakova=2021} similarly as in \cite{Gutik-Mykhalenych=2020}. In \cite{Gutik-Pozdniakova=2021} it is proven that $\boldsymbol{B}_{\mathbb{Z}}^{\mathscr{F}}$ is a combinatorial inverse semigroup. Green's relations, the natural partial order on the semigroup $\boldsymbol{B}_{\mathbb{Z}}^{\mathscr{F}}$ and its set of idempotents are described.
Here, the criteria when the semigroup $\boldsymbol{B}_{\mathbb{Z}}^{\mathscr{F}}$ is simple, $0$-simple, bisimple, $0$-bisimple, is isomorphic to the extended bicyclic semigroup,  are derived. In particularly in \cite{Gutik-Pozdniakova=2021} it is proved that the semigroup $\boldsymbol{B}_{\mathbb{Z}}^{\mathscr{F}}$ is isomorphic to the semigrpoup of ${\omega}{\times}{\omega}$-matrix units if and only if $\mathscr{F}$ consists of a singleton set and the empty set, and $\boldsymbol{B}_{Z}^{\mathscr{F}}$ is isomorphic to the extended bicyclic semigroup if and only if $\mathscr{F}$ consists of a non-empty inductive subset of $\omega$.
Also, in \cite{Gutik-Pozdniakova=2021} it is proved that in the case when the family $\mathscr{F}$ consists of all singletons of $\omega$ and the empty set, the semigroup $\boldsymbol{B}_{\mathbb{Z}}^{\mathscr{F}}$ is isomorphic to the Brandt $\lambda$-extension of the semilattice $(\omega,\min)$, where $(\omega,\min)$ is the set $\omega$ with the semilattice operation $x\cdot y=\min\{x,y\}$.

It is well-known that every automorphism of the bicyclic monoid $\boldsymbol{B}_{\omega}$  is the identity self-map of $\boldsymbol{B}_{\omega}$ \cite{Clifford-Preston-1961}, and hence the group $\mathbf{Aut}(\boldsymbol{B}_{\omega})$ of automorphisms of $\boldsymbol{B}_{\omega}$ is trivial. The group $\mathbf{Aut}(\boldsymbol{B}_{\mathbb{Z}})$ of automorphisms of the extended bicyclic semigroup $\boldsymbol{B}_{\mathbb{Z}}$ is established in \cite{Gutik-Maksymyk-2017} and there it is proved that $\mathbf{Aut}(\boldsymbol{B}_{\mathbb{Z}})$ is isomorphic to the additive group of integers $\mathbb{Z}(+)$. In the paper \cite{Gutik-Pozdniakova=2023} we prove that for any family $\mathscr{F}$ of nonempty inductive subsets of $\omega$ the group $\mathbf{Aut}(\boldsymbol{B}_{\mathbb{Z}}^{\mathscr{F}})$ of automorphisms of the semigroup $\boldsymbol{B}_{\mathbb{Z}}^{\mathscr{F}}$ is isomorphic to the additive group of integers.

In \cite{Gutik-Prokhorenkova-Sekh=2021} the semigroups of endomorphisms of the bicyclic  semigroup and  the extended bicyclic semigroup are described.
All types of monoid endomorphisms of the monoid $\boldsymbol{B}_{\omega}^{\mathscr{F}^2}$ for two-element family $\mathscr{F}^2$ of nonempty inductive subsets of $\omega$ are described in \cite{Gutik-Pozdniakova=2022, Gutik-Pozdniakova=2023a, Gutik-Pozdniakova=2024}.

This paper is a continuation of \cite{Gutik-Pozdniakova=2021, Gutik-Pozdniakova=2023}. We describe injective endomorphisms of the semigroup $\boldsymbol{B}_{Z\mathbb{}}^{\mathscr{F}^2}$ with the two-element family $\mathscr{F}^2$ of inductive nonempty subsets of $\omega$. In particular we show that every injective endomorphism $\mathfrak{e}$ of $\boldsymbol{B}_{Z\mathbb{}}^{\mathscr{F}^2}$ is presented in the form $\mathfrak{e}=\mathfrak{e}_0\mathfrak{a}$, where $\mathfrak{e}_0$ is an injective  $(0,0,[0))$-endomorphism of $\boldsymbol{B}_{\mathbb{Z}}^{\mathscr{F}^2}$ and $\mathfrak{a}$ is an automorphism $\mathfrak{a}$ of $\boldsymbol{B}_{\mathbb{Z}}^{\mathscr{F}^2}$. Also we describe all injective  $(0,0,[0))$-endomorphisms $\mathfrak{e}_0$ of $\boldsymbol{B}_{\mathbb{Z}}^{\mathscr{F}^2}$, i.e., such that $(0,0,[0))\mathfrak{e}_0=(0,0,[0))$.

Later we assume that an ${\omega}$-closed family $\mathscr{F}^2$ consists of two inductive nonempty subsets of $\omega$.

\section{Endomorphisms of the semigroup $\boldsymbol{B}_{Z\mathbb{}}^{\mathscr{F}^2}$ with the fixed point $(0,0,[0))$}\label{section-2}

If $\mathscr{F}$ is an arbitrary $\omega$-closed family of inductive subsets in $\mathscr{P}(\omega)$ and $[s)\in\mathscr{F}$
for some $s\in\mathbb{Z}$ then
\begin{equation*}
  \boldsymbol{B}_{Z\mathbb{}}^{\{[s)\}}=\{(i,j,[s))\colon i,j\in\mathbb{Z}\}
\end{equation*}
is a subsemigroup of $\boldsymbol{B}_{\mathbb{Z}}^{\mathscr{F}}$  and by Proposition 5 of \cite{Gutik-Pozdniakova=2021} the semigroup $\boldsymbol{B}_{Z\mathbb{}}^{\{[s)\}}$ is isomorphic to the extended bicyclic semigroup.

\begin{remark}\label{remark-2.1}
By Proposition 1 of \cite{Gutik-Pozdniakova=2023} for any $\omega$-closed family $\mathscr{F}$ of inductive subsets in $\mathscr{P}(\omega)$ there exists an $\omega$-closed family $\mathscr{F}^*$ of inductive subsets in $\mathscr{P}(\omega)$ such that $[0)\in \mathscr{F}^*$ and the semigroups $\boldsymbol{B}_{\mathbb{Z}}^{\mathscr{F}}$ and $\boldsymbol{B}_{\mathbb{Z}}^{\mathscr{F}^*}$ are isomorphic. Hence without loss of generality we may assume that the family $\mathscr{F}^2$ contains the set $[0)$.
\end{remark}

An endomorphism $\mathfrak{e}$ of the semigroup $\boldsymbol{B}_{\mathbb{Z}}^{\mathscr{F}^2}$ is called \emph{$(0,0,[0))$-endomorphism} if $(0,0,[0))\mathfrak{e}=(0,0,[0))$.

\begin{remark}\label{remark-2.2}
Theorem~1 of \cite{Gutik-Pozdniakova=2022} state that for every injective monoid endomorphism $\mathfrak{e}$ of the monoid $\boldsymbol{B}_{\omega}^{\mathscr{F}^2}$ only one of the following conditions holds:
\begin{enumerate}
  \item there exist a positive integer $k$ and $p\in\{0,\ldots,k-1\}$ such that $\mathfrak{e}=\alpha_{k,p}$, where the mapping $\alpha_{k,p}$ defined by the formula
      \begin{align*}
        (i,j,[0))\alpha_{k,p}&=(ki,kj,[0)), \\
        (i,j,[1))\alpha_{k,p}&=(p+ki,p+kj,[1)),
      \end{align*}
      $i,j\in\omega$;
  \item there exist a positive integer $k\geqslant 2$ and $p\in\{1,\ldots,k-1\}$ such that $\mathfrak{e}=\beta_{k,p}$, where the mapping $\beta_{k,p}$ defined by the formula
      \begin{align*}
        (i,j,[0))\beta_{k,p}&=(ki,kj,[0)), \\
        (i,j,[1))\beta_{k,p}&=(p+ki,p+kj,[0)),
      \end{align*}
      $i,j\in\omega$.
\end{enumerate}
\end{remark}

For any integer $k$ we define
\begin{equation*}
  \boldsymbol{B}_{\mathbb{Z}}^{\mathscr{F}^2}(k,k,0)=(k, k, [0))\cdot  \boldsymbol{B}_{\mathbb{Z}}^{\mathscr{F}^2}\cdot(k, k, [0)).
\end{equation*}
By Proposition 2~\cite{Gutik-Pozdniakova=2023},  $\boldsymbol{B}_{\mathbb{Z}}^{\mathscr{F}^2}(k,k,0)$ is a subsemigroup of $\boldsymbol{B}_{\mathbb{Z}}^{\mathscr{F}^2}$ which is isomorphic to $\boldsymbol{B}_{\omega}^{\mathscr{F}^2}$.

Fix an arbitrary positive integer $k$ and any $p\in\{0,\ldots,k-1\}$. For all $i,j\in\mathbb{Z}$ we denote the transformation $\alpha_{k,p}$ of the semigroup $\boldsymbol{B}_{\mathbb{Z}}^{\mathscr{F}^2}$ in the following way
\begin{align*}
        (i,j,[0))\alpha_{k,p}&=(ki,kj,[0)), \\
        (i,j,[1))\alpha_{k,p}&=(p+ki,p+kj,[1)),
      \end{align*}

\begin{lemma}\label{lemma-2.3}
For an arbitrary positive integer $k$ and any $p\in\{0,\ldots,k-1\}$ the map $\alpha_{k,p}$ is an injective endomorphism of the semigroup $\boldsymbol{B}_{\mathbb{Z}}^{\mathscr{F}^2}$.
\end{lemma}

\begin{proof}
 By by Proposition 5 of \cite{Gutik-Pozdniakova=2021} the subsemigroups $\boldsymbol{B}_{\mathbb{Z}}^{\{[0)\}}$ and $\boldsymbol{B}_{\mathbb{Z}}^{\{[1)\}}$ are isomorphic to the extended bicyclic semigroup. By Proposition of \cite{Gutik-Prokhorenkova-Sekh=2021} we have that the restrictions of $\alpha_{k,p}$ onto the subsemigroups $\boldsymbol{B}_{\mathbb{Z}}^{\{[0)\}}$ and $\boldsymbol{B}_{\mathbb{Z}}^{\{[1)\}}$ are endomorphisms of $\boldsymbol{B}_{\mathbb{Z}}^{\{[0)\}}$ and $\boldsymbol{B}_{\mathbb{Z}}^{\{[1)\}}$, respectively. This implies that for all $i,j,s,t\in \mathbb{Z}$ the following equalities hold
\begin{align*}
  ((i,j,[0))\cdot(s,t,[0)))\alpha_{k,p}&=(i,j,[0))\alpha_{k,p}\cdot (s,t,[0))\alpha_{k,p}, \\
  ((i,j,[1))\cdot(s,t,[1)))\alpha_{k,p}&=(i,j,[1))\alpha_{k,p}\cdot (s,t,[1))\alpha_{k,p}.
\end{align*}

For any $i,j,p,q\in \mathbb{Z}$ we have that
\begin{align*}
  ((i,j,[0))\cdot(s,t,[1)))\alpha_{k,p}&=
  \left\{
    \begin{array}{ll}
      (i+s-j,t,(j-s+[0))\cap[1))\alpha_{k,p}, & \hbox{if~} j<s;\\
      (i,t,[0)\cap[1))\alpha_{k,p},           & \hbox{if~} j=s;\\
      (i,j+t-s,[0)\cap(s-j+[1)))\alpha_{k,p}, & \hbox{if~} j>s
    \end{array}
  \right.
   =\\
   &=
   \left\{
    \begin{array}{ll}
      (i+s-j,t,[1))\alpha_{k,p}, & \hbox{if~} j<s;\\
      (i,t,[1))\alpha_{k,p},     & \hbox{if~} j=s;\\
      (i,j+t-s,[0))\alpha_{k,p}, & \hbox{if~} j>s
    \end{array}
  \right.
   =\\
   &=
   \left\{
    \begin{array}{ll}
      (p+k(i+s-j),p+kt,[1)), & \hbox{if~} j<s;\\
      (p+ki,p+kt,[1)),       & \hbox{if~} j=s;\\
      (ki,k(j+t-s),[0)),     & \hbox{if~} j>s,
    \end{array}
  \right.
\end{align*}
\begin{align*}
  (i,j,[0))&\alpha_{k,p}\cdot (s,t,[1))\alpha_{k,p}=(ki,kj,[0))\cdot(p+ks,p+kt,[1))= \\
   &=
   \left\{
     \begin{array}{ll}
       (ki+p+ks-kj,p+kt,(kj-p-ks+[0))\cap[1)), & \hbox{if~} kj<p+ks;\\
       (ki,p+kt,[0)\cap[1)),                   & \hbox{if~} kj=p+ks;\\
       (ki,kj+p+kt-p-ks,[0)\cap(p+ks-kj+[1))), & \hbox{if~} kj>p+ks
     \end{array}
   \right.
   = \\
   &=
   \left\{
     \begin{array}{ll}
       (p+k(i+s-j),p+kt,[1)),              & \hbox{if~} kj<p+ks;\\
       (ki,p+kt,[1)),                      & \hbox{if~} kj=p+ks;\\
       (ki,k(j+t-s),[0)), & \hbox{if~} kj>p+ks
     \end{array}
   \right.
   = \\
   &=
   \left\{
     \begin{array}{ll}
       (p+k(i+s-j),p+kt,[1)), & \hbox{if~} kj<ks;\\
       (p+ki,p+kt,[1)),       & \hbox{if~} kj=ks;\\
       (ki,kt,[1)),                   & \hbox{if~} kj=p+ks \hbox{~and~} p=0;\\
       \texttt{vagueness},                   & \hbox{if~} kj=p+ks \hbox{~and~} p\neq0;\\
       (ki,k(j+t-s),[0)), & \hbox{if~} kj>ks
     \end{array}
   \right.
   = \\
   &=
   \left\{
    \begin{array}{ll}
      (p+k(i+s-j),p+kt,[1)), & \hbox{if~} j<s;\\
      (p+ki,p+kt,[1)),       & \hbox{if~} j=s;\\
      (ki,k(j+t-s),[0)),     & \hbox{if~} j>s,
    \end{array}
  \right.
\end{align*}
and
\begin{align*}
  ((i,j,[1))\cdot(s,t,[0)))\alpha_{k,p}&=
    \left\{
     \begin{array}{ll}
       (i+s-j,t,(j-s+[1))\cap[0))\alpha_{k,p}, & \hbox{if~} j<s;\\
       (i,t,[1)\cap[0))\alpha_{k,p},           & \hbox{if~} j=s;\\
       (i,j+t-s,[1)\cap(s-j+[0)))\alpha_{k,p}, & \hbox{if~} j>s
     \end{array}
   \right.
   = 
\end{align*}
\begin{align*}   
   &=
   \left\{
     \begin{array}{ll}
       (i+s-j,t,[0))\alpha_{k,p}, & \hbox{if~} j<s;\\
       (i,t,[1))\alpha_{k,p},     & \hbox{if~} j=s;\\
       (i,j+t-s,[1))\alpha_{k,p}, & \hbox{if~} j>s
     \end{array}
   \right.
   = \\
   &=
   \left\{
     \begin{array}{ll}
       (k(i+s-j),kt,[0)),     & \hbox{if~} j<s;\\
       (p+ki,p+kt,[1)),       & \hbox{if~} j=s;\\
       (p+ki,p+k(j+t-s),[1)), & \hbox{if~} j>s,
     \end{array}
   \right.
\end{align*}
\begin{align*}
  (i,j,[1))&\alpha_{k,p}\cdot (s,t,[0))\alpha_{k,p}=(p+ki,p+kj,[1))\cdot(ks,kt,[0))= \\
   &=
   \left\{
     \begin{array}{ll}
       (p+ki+ks-p-kj,kt,(p+kj-ks+[1))\cap[0)), & \hbox{if~} p+kj<ks; \\
       (p+ki,kt,[1)\cap[0)),                   & \hbox{if~} p+kj=ks; \\
       (p+ki,p+kj+kt-ks,[1)\cap(ks-p-kj+[0))), & \hbox{if~} p+kj>ks
     \end{array}
   \right.
   = \\
   &=
   \left\{
     \begin{array}{ll}
       (k(i+s-j),kt,[0)),     & \hbox{if~} p+kj<ks; \\
       (p+ki,kt,[1)),         & \hbox{if~} p+kj=ks; \\
       (p+ki,p+k(j+t-s),[1)), & \hbox{if~} p+kj>ks
     \end{array}
   \right.
   = \\
   &=
   \left\{
     \begin{array}{ll}
       (k(i+s-j),kt,[0)),     & \hbox{if~} kj<ks; \\
       (ki,kt,[1)),           & \hbox{if~} p+kj=ks \hbox{~and~} p=0; \\
       \texttt{vagueness},    & \hbox{if~} p+kj=ks \hbox{~and~} p\neq0; \\
       (p+ki,p+kt,[1)),       & \hbox{if~} kj=ks; \\
       (p+ki,p+k(j+t-s),[1)), & \hbox{if~} kj>ks
     \end{array}
   \right.
   = \\
   &=
   \left\{
     \begin{array}{ll}
       (k(i+s-j),kt,[0)),     & \hbox{if~} j<s;\\
       (p+ki,p+kt,[1)),       & \hbox{if~} j=s;\\
       (p+ki,p+k(j+t-s),[1)), & \hbox{if~} j>s,
     \end{array}
   \right.
\end{align*}
because $p\in\{0,\ldots,k-1\}$. Thus, $\alpha_{k,p}$ is an endomorphism of the semigroup $\boldsymbol{B}_{\mathbb{Z}}^{\mathscr{F}^2}$.
\end{proof}

Fix an arbitrary positive integer $k\geqslant 2$ and any $p\in\{1,\ldots,k-1\}$. For all $i,j\in\mathbb{Z}$ we define the transformation $\beta_{k,p}$  of the semigroup $\boldsymbol{B}_{\mathbb{Z}}^{\mathscr{F}^2}$ in the following way
\begin{align*}
  (i,j,[0))\beta_{k,p}&=(ki,kj,[0)), \\
  (i,j,[1))\beta_{k,p}&=(p+ki,p+kj,[0)).
\end{align*}
It is obvious that $\beta_{k,p}$ is an injective transformation of the semigroup $\boldsymbol{B}_{\mathbb{Z}}^{\mathscr{F}^2}$.

\begin{lemma}\label{lemma-2.4}
For an arbitrary positive integer $k\geqslant 2$ and any $p\in\{1,\ldots,k-1\}$ the map $\beta_{k,p}$ is an injective endomorphism of the semigroup $\boldsymbol{B}_{\mathbb{Z}}^{\mathscr{F}^2}$.
\end{lemma}

\begin{proof}
By by Proposition 5 of \cite{Gutik-Pozdniakova=2021} the subsemigroups $\boldsymbol{B}_{\mathbb{Z}}^{\{[0)\}}$ and $\boldsymbol{B}_{\mathbb{Z}}^{\{[1)\}}$ are isomorphic to the extended bicyclic semigroup. By Lemma~3 of \cite{Gutik-Prokhorenkova-Sekh=2021}, the restriction of $\beta_{k,p}$ onto the subsemigroup $\boldsymbol{B}_{\mathbb{Z}}^{\{[0)\}}$ is an endomorphisms of $\boldsymbol{B}_{\mathbb{Z}}^{\{[0)\}}$, and the restriction of $\beta_{k,p}$ onto the subsemigroup $\boldsymbol{B}_{\mathbb{Z}}^{\{[1)\}}$ is a homomorphisms of $\boldsymbol{B}_{\mathbb{Z}}^{\{[1)\}}$ into $\boldsymbol{B}_{\mathbb{Z}}^{\{[0)\}}$. This implies that for all $i,j,s,t\in \mathbb{Z}$ the following equalities hold
\begin{align*}
  ((i,j,[0))\cdot(s,t,[0)))\beta_{k,p}&=(i,j,[0))\beta_{k,p}\cdot (s,t,[0))\beta_{k,p}, \\
  ((i,j,[1))\cdot(s,t,[1)))\beta_{k,p}&=(i,j,[1))\beta_{k,p}\cdot (s,t,[1))\beta_{k,p}.
\end{align*}

For any $i,j,p,q\in \mathbb{Z}$ we have that
\begin{align*}
  ((i,j,[0))\cdot(s,t,[1)))\beta_{k,p}&=
  \left\{
    \begin{array}{ll}
      (i+s-j,t,(j-s+[0))\cap[1))\beta_{k,p}, & \hbox{if~} j<s;\\
      (i,t,[0)\cap[1))\beta_{k,p},           & \hbox{if~} j=s;\\
      (i,j+t-s,[0)\cap(s-j+[1)))\beta_{k,p}, & \hbox{if~} j>s
    \end{array}
  \right.
   =
\end{align*}
\begin{align*} 
   &=
   \left\{
    \begin{array}{ll}
      (i+s-j,t,[1))\beta_{k,p}, & \hbox{if~} j<s;\\
      (i,t,[1))\beta_{k,p},     & \hbox{if~} j=s;\\
      (i,j+t-s,[0))\beta_{k,p}, & \hbox{if~} j>s
    \end{array}
  \right.
   =\\
   &=
   \left\{
    \begin{array}{ll}
      (p+k(i+s-j),p+kt,[0)), & \hbox{if~} j<s;\\
      (p+ki,p+kt,[0)),       & \hbox{if~} j=s;\\
      (ki,k(j+t-s),[0)),     & \hbox{if~} j>s,
    \end{array}
  \right.
\end{align*}
\begin{align*}
  (i,j,[0))&\beta_{k,p}\cdot (s,t,[1))\beta_{k,p}=(ki,kj,[0))\cdot(p+ks,p+kt,[0))= \\
   &=
   \left\{
     \begin{array}{ll}
       (ki+p+ks-kj,p+kt,(kj-p-ks+[0))\cap[0)), & \hbox{if~} kj<p+ks;\\
       (ki,p+kt,[0)\cap[0)),                   & \hbox{if~} kj=p+ks;\\
       (ki,kj+p+kt-p-ks,[0)\cap(p+ks-kj+[0))), & \hbox{if~} kj>p+ks
     \end{array}
   \right.
   = \\
   &=
   \left\{
     \begin{array}{ll}
       (p+k(i+s-j),p+kt,[0)),              & \hbox{if~} kj<p+ks;\\
       (ki,p+kt,[0)),                      & \hbox{if~} kj=p+ks;\\
       (ki,k(j+t-s),[0)),                  & \hbox{if~} kj>p+ks
     \end{array}
   \right.
   = \\
   &=
   \left\{
     \begin{array}{ll}
       (p+k(i+s-j),p+kt,[0)), & \hbox{if~} kj<ks;\\
       (p+ki,p+kt,[0)),       & \hbox{if~} kj=ks;\\
       (ki,kt,[0)),           & \hbox{if~} kj=p+ks \hbox{~and~} p=0;\\
       \texttt{vagueness},    & \hbox{if~} kj=p+ks \hbox{~and~} p\neq0;\\
       (ki,k(j+t-s),[0))),    & \hbox{if~} kj>ks
     \end{array}
   \right.
   = \\
   &=
   \left\{
    \begin{array}{ll}
      (p+k(i+s-j),p+kt,[0)), & \hbox{if~} j<s;\\
      (p+ki,p+kt,[0)),       & \hbox{if~} j=s;\\
      (ki,k(j+t-s),[0)),     & \hbox{if~} j>s,
    \end{array}
  \right.
\end{align*}
and
\begin{align*}
  ((i,j,[1))\cdot(s,t,[0)))\beta_{k,p}&=
    \left\{
     \begin{array}{ll}
       (i+s-j,t,(j-s+[1)\cap[0))\beta_{k,p},  & \hbox{if~} j<s;\\
       (i,t,[1)\cap[0))\beta_{k,p},           & \hbox{if~} j=s;\\
       (i,j+t-s,[1)\cap(s-j+[0)))\beta_{k,p}, & \hbox{if~} j>s
     \end{array}
   \right.
   = \\
   &=
   \left\{
     \begin{array}{ll}
       (i+s-j,t,[0))\beta_{k,p}, & \hbox{if~} j<s;\\
       (i,t,[1))\beta_{k,p},     & \hbox{if~} j=s;\\
       (i,j+t-s,[1))\beta_{k,p}, & \hbox{if~} j>s
     \end{array}
   \right.
   = \\
   &=
   \left\{
     \begin{array}{ll}
       (k(i+s-j),kt,[0)),     & \hbox{if~} j<s;\\
       (p+ki,p+kt,[0)),       & \hbox{if~} j=s;\\
       (p+ki,p+k(j+t-s),[0)), & \hbox{if~} j>s,
     \end{array}
   \right.
\end{align*}
\begin{align*}
  (i,j,[1))&\beta_{k,p}\cdot (s,t,[0))\beta_{k,p}=(p+ki,p+kj,[0))\cdot(ks,kt,[0))= \\
   &=
   \left\{
     \begin{array}{ll}
       (p+ki+ks-p-kj,kt,(p+kj-ks+[0))\cap[0)), & \hbox{if~} p+kj<ks; \\
       (p+ki,kt,[0)\cap[0)),                   & \hbox{if~} p+kj=ks; \\
       (p+ki,p+kj+kt-ks,[0)\cap(ks-p-kj+[0))), & \hbox{if~} p+kj>ks
     \end{array}
   \right.
   = \\
   &=
   \left\{
     \begin{array}{ll}
       (k(i+s-j),kt,[0)),     & \hbox{if~} p+kj<ks; \\
       (p+ki,kt,[0)),         & \hbox{if~} p+kj=ks; \\
       (p+ki,p+k(j+t-s),[0)), & \hbox{if~} p+kj>ks
     \end{array}
   \right.
   = 
\\
   &=
   \left\{
     \begin{array}{ll}
       (k(i+s-j),kt,[0)),     & \hbox{if~} kj<ks; \\
       (ki,kt,[0)),           & \hbox{if~} p+kj=ks \hbox{~and~} p=0; \\
       \texttt{vagueness},    & \hbox{if~} p+kj=ks \hbox{~and~} p\neq0; \\
       (p+ki,p+kt,[0)),       & \hbox{if~} kj=ks; \\
       (p+ki,p+k(j+t-s),[0)), & \hbox{if~} kj>ks
     \end{array}
   \right.
   = 
\end{align*}
\begin{align*}     
   &=
   \left\{
     \begin{array}{ll}
       (k(i+s-j),kt,[0)),     & \hbox{if~} j<s;\\
       (p+ki,p+kt,[0)),       & \hbox{if~} j=s;\\
       (p+ki,p+k(j+t-s),[0)), & \hbox{if~} j>s,
     \end{array}
   \right.
\end{align*}
because $p\in\{1,\ldots,k-1\}$. Thus, $\beta_{k,p}$ is an endomorphism of the semigroup $\boldsymbol{B}_{\mathbb{Z}}^{\mathscr{F}^2}$.
\end{proof}

\begin{lemma}\label{lemma-2.5}
Let $\mathfrak{e}$ be a  $(0,0,[0))$-endomorphism of the semigroup $\boldsymbol{B}_{\mathbb{Z}}^{\mathscr{F}^2}$. Then  there exists a non-negative integer $n$ such that $(i,j,[0))\mathfrak{e}=(ni,nj,[0))$ for all $i,j\in\mathbb{Z}$.
\end{lemma}

\begin{proof}
By by Proposition 5 of \cite{Gutik-Pozdniakova=2021} the subsemigroup $\boldsymbol{B}_{\mathbb{Z}}^{\{[0)\}}$ is isomorphic to the extended bicyclic semigroup. By Lemma~3 of \cite{Gutik-Prokhorenkova-Sekh=2021}, the restriction of the transformation $\mathfrak{e}$ onto the subsemigroup $\boldsymbol{B}_{\mathbb{Z}}^{\{[0)\}}$ is an endomorphisms of $\boldsymbol{B}_{\mathbb{Z}}^{\{[0)\}}$. Then Lemma~5 of \cite{Gutik-Prokhorenkova-Sekh=2021} implies the statement of the lemma.
\end{proof}

\begin{theorem}\label{theorem-2.6}
Let $\mathfrak{e}$ be an injective $(0,0,[0))$-endomorphism of the semigroup $\boldsymbol{B}_{\mathbb{Z}}^{\mathscr{F}^2}$. Then one of the following conditions holds:
\begin{enumerate}
  \item\label{theorem-2.3(1)} there exist a positive integer $k$ and $p\in\{0,\ldots,k-1\}$ such that $\mathfrak{e}=\alpha_{k,p}$;
  \item\label{theorem-2.3(2)} there exist a positive integer $k\geqslant 2$ and $p\in\{1,\ldots,k-1\}$ such that $\mathfrak{e}=\beta_{k,p}$.
\end{enumerate}
\end{theorem}

\begin{proof}
It is obvious that for any  $(i,j,[l))\in \boldsymbol{B}_{\mathbb{Z}}^{\mathscr{F}^2}$, $l=0,1$, there exists a non-negative integer $n$ such that  $(i,j,[l))\in \boldsymbol{B}_{\mathbb{Z}}^{\mathscr{F}^2}(-n,-n,0)$. This implies the equality
\begin{equation*}
  \boldsymbol{B}_{\mathbb{Z}}^{\mathscr{F}^2}=\bigcup_{n\in\omega}\boldsymbol{B}_{\mathbb{Z}}^{\mathscr{F}^2}(-n,-n,0).
\end{equation*}
Also by the semigroup operation of $\boldsymbol{B}_{\mathbb{Z}}^{\mathscr{F}^2}$ for $m,n\in \omega$ we have that $\boldsymbol{B}_{\mathbb{Z}}^{\mathscr{F}^2}(-n,-n,0)\subseteq \boldsymbol{B}_{\mathbb{Z}}^{\mathscr{F}^2}(-m,-m,0)$ if and only if $m\geqslant n$.

Since $\mathfrak{e}$ is an injective $(0,0,[0))$-endomorphism of the semigroup $\boldsymbol{B}_{\mathbb{Z}}^{\mathscr{F}^2}$, Lemma~\ref{lemma-2.5} implies that there exists a positive integer $k$ such that $(i,j,[0))\mathfrak{e}=(ki,kj,[0))$ for all $i,j\in\mathbb{Z}$.

Fix an arbitrary positive integer $n$. By Proposition 2~\cite{Gutik-Pozdniakova=2023},  $\boldsymbol{B}_{\mathbb{Z}}^{\mathscr{F}^2}(-n,-n,0)$ is a subsemigroup of $\boldsymbol{B}_{\mathbb{Z}}^{\mathscr{F}^2}$ which is isomorphic to $\boldsymbol{B}_{\omega}^{\mathscr{F}^2}$. This implies that the semigroups $\boldsymbol{B}_{\mathbb{Z}}^{\mathscr{F}^2}(-n,-n,0)$ and $\boldsymbol{B}_{\mathbb{Z}}^{\mathscr{F}^2}(0,0,0)$ are isomorphic. By Corollary~2 from \cite{Gutik-Mykhalenych=2021} every automorphism of the semigroup $\boldsymbol{B}_{\omega}^{\mathscr{F}^2}$ is the identity map, and hence every automorphism of the semigroup $\boldsymbol{B}_{\mathbb{Z}}^{\mathscr{F}^2}(0,0,0)$ is the identity map, too.

We define the isomorphism $\mathfrak{I}_0^{-n}\colon \boldsymbol{B}_{\mathbb{Z}}^{\mathscr{F}^2}(-n,-n,0)\to \boldsymbol{B}_{\mathbb{Z}}^{\mathscr{F}^2}(0,0,0)$ by the formula
\begin{equation*}
  (i-n,j-n,[s))\mathfrak{I}_0^{-n}=(i,j,[s)),
\end{equation*}
for any positive integers $i,j$ and $s\in\{0,1\}$. The above arguments imply that so defined isomorphism is unique. Hence we have that for any injective endomorphism $\mathfrak{e}_{-n}$ of the semigroup $\boldsymbol{B}_{\mathbb{Z}}^{\mathscr{F}^2}(-n,-n,0)$ there exists an injective endomorphism $\mathfrak{e}_{0}$ of the semigroup $\boldsymbol{B}_{\mathbb{Z}}^{\mathscr{F}^2}(0,0,0)$ such that the following diagram
\begin{equation*}
\xymatrix{
 \boldsymbol{B}_{\mathbb{Z}}^{\mathscr{F}^2}(-n,-n,0)\ar[rr]^{\mathfrak{e}_{-n}}\ar[dd]_{\mathfrak{I}_0^{-n}}  & & \boldsymbol{B}_{\mathbb{Z}}^{\mathscr{F}^2}(-n,-n,0)\ar[dd]^{\mathfrak{I}_0^{-n}}\\
 & &\\
 \boldsymbol{B}_{\mathbb{Z}}^{\mathscr{F}^2}(0,0,0)\ar@{-->}[rr]_{\mathfrak{e}_{0}} &&  \boldsymbol{B}_{\mathbb{Z}}^{\mathscr{F}^2}(0,0,0)}
\end{equation*}
is commutative. Hence, by Remark~\ref{remark-2.2} we have that for any injective endomorphism $\mathfrak{e}_{-n}$ of the semigroup $\boldsymbol{B}_{\mathbb{Z}}^{\mathscr{F}^2}(-n,-n,0)$ one of the following conditions holds:
\begin{enumerate}
  \item there exist a positive integer $k$ and $p\in\{0,\ldots,k-1\}$ such that $\mathfrak{e}_0=\alpha_{k,p}$;
  \item there exist a positive integer $k\geqslant 2$ and $p\in\{1,\ldots,k-1\}$ such that $\mathfrak{e}_0=\beta_{k,p}$.
\end{enumerate}

If $\mathfrak{e}_0=\alpha_{k,p}$ then
\begin{align*}
        (i-n,j-n,[0))\mathfrak{e}_{-n}&=(((i-n,j-n,[0))\mathfrak{I}_0^{-n})\alpha_{k,p})(\mathfrak{I}_0^{-n})^{-1}=\\
                                   &=((i,j,[0))\alpha_{k,p})(\mathfrak{I}_0^{-n})^{-1}=\\
                                   &=(ki,kj,[0))(\mathfrak{I}_0^{-n})^{-1}=\\
                                   &=(ki-n,kj-n,[0))
\end{align*}
and
\begin{align*}
        (i-n,j-n,[1))\mathfrak{e}_{-n}&=(((i-n,j-n,[1))\mathfrak{I}_0^{-n})\alpha_{k,p})(\mathfrak{I}_0^{-n})^{-1}=\\
                                   &=((i,j,[0))\alpha_{k,p})(\mathfrak{I}_0^{-n})^{-1}=\\
                                   &=(p+ki,p+kj,[1))(\mathfrak{I}_0^{-n})^{-1}=\\
                                   &=(p+ki-n,p+kj-n,[1)),
      \end{align*}
for any positive integers $i,j$.

If $\mathfrak{e}_0=\beta_{k,p}$ then
\begin{align*}
        (i-n,j-n,[0))\mathfrak{e}_{-n}&=(((i-n,j-n,[0))\mathfrak{I}_0^{-n})\beta_{k,p})(\mathfrak{I}_0^{-n})^{-1}=\\
                                   &=((i,j,[0))\beta_{k,p})(\mathfrak{I}_0^{-n})^{-1}=\\
                                   &=(ki,kj,[0))(\mathfrak{I}_0^{-n})^{-1}=\\
                                   &=(ki-n,kj-n,[0))
\end{align*}
and
\begin{align*}
        (i-n,j-n,[1))\mathfrak{e}_{-n}&=(((i-n,j-n,[1))\mathfrak{I}_0^{-n})\beta_{k,p})(\mathfrak{I}_0^{-n})^{-1}=\\
                                   &=((i,j,[0))\beta_{k,p})(\mathfrak{I}_0^{-n})^{-1}=\\
                                   &=(p+ki,p+kj,[0))(\mathfrak{I}_0^{-n})^{-1}=\\
                                   &=(p+ki-n,p+kj-n,[0)),
      \end{align*}
for any positive integers $i,j$.

This completes the proof of the theorem.
\end{proof}

\section{On injective endomorphisms of the semigroup $\boldsymbol{B}_{\mathbb{Z}}^{\mathscr{F}^2}$}\label{section-3}

\begin{remark}\label{remark-3.1}
\begin{enumerate}
  \item\label{remark-3.1(1)} By Theorem~1 of \cite{Gutik-Pozdniakova=2023} every $(0,0,[0))$-automorphism of the semigroup $\boldsymbol{B}_{\mathbb{Z}}^{\mathscr{F}^2}$ is the identity map.
  \item\label{remark-3.1(2)} For every integer $s$ the map $\mathfrak{h}_s\colon \boldsymbol{B}_{\mathbb{Z}}^{\mathscr{F}^2}\to\boldsymbol{B}_{\mathbb{Z}}^{\mathscr{F}^2}$, $(i,j,[p))\mapsto(i+s,j+s,[q))$, $i,j\in\mathbb{Z}$, $q\in\{0,1\}$, is an automorphism of the semigroup $\boldsymbol{B}_{\mathbb{Z}}^{\mathscr{F}^2}$ (Proposition~6 of \cite{Gutik-Pozdniakova=2023}).
  \item\label{remark-3.1(3)} The map $\widetilde{\mathfrak{a}}\colon \boldsymbol{B}_{\mathbb{Z}}^{\mathscr{F}^2}\to\boldsymbol{B}_{\mathbb{Z}}^{\mathscr{F}^2}$, $(i,j,[p))\mapsto(i+q,j+q,[1-q))$, $i,j\in\mathbb{Z}$, $q\in\{0,1\}$, is an automorphism of the semigroup $\boldsymbol{B}_{\mathbb{Z}}^{\mathscr{F}^2}$ (Lemma~2 of \cite{Gutik-Pozdniakova=2023}), and moreover $\widetilde{\mathfrak{a}}^2=\mathfrak{h}_1$.
\end{enumerate}
\end{remark}

\begin{lemma}\label{lemma-3.2}
For any endomorphism $\mathfrak{e}$ of the semigroup $\boldsymbol{B}_{\mathbb{Z}}^{\mathscr{F}^2}$ there exists an automorphism $\mathfrak{a}$ of $\boldsymbol{B}_{\mathbb{Z}}^{\mathscr{F}^2}$ such that $(0,0,[0))\mathfrak{e}=(0,0,[0))\mathfrak{a}$. Moreover, $\mathfrak{a}=\widetilde{\mathfrak{a}}^{2s}=\mathfrak{h}_s$ in the case when $(0,0,[0))\mathfrak{e}=(s,s,[0))$, and $\mathfrak{a}=\widetilde{\mathfrak{a}}^{2s+1}=\mathfrak{h}_{s}\widetilde{\mathfrak{a}}$ in the case when $(0,0,[0))\mathfrak{e}=(s,s,[1))$ for some integer $s$.
\end{lemma}

\begin{proof}
Since any homomorphic image of an idempotent is again an idempotent, by Lemma~1 of \cite{Gutik-Pozdniakova=2021} there exist an integer $s$ and $q\in\{0,1\}$ such that $(0,0,[0))\mathfrak{e}=(s,s,[q))$. Simple verifications and Re\-mark~\ref{remark-3.1} imply that $(0,0,[0))\widetilde{\mathfrak{a}}^{2s}=0,0,[0))=(s,s,[0))$ and  $(0,0,[0))\widetilde{\mathfrak{a}}^{2s+1}=(0,0,[0))(\mathfrak{h}_{s}\widetilde{\mathfrak{a}})=(s,s,[1))$.
\end{proof}

\begin{theorem}\label{theorem-3.3}
For any endomorphism $\mathfrak{e}$ of the semigroup $\boldsymbol{B}_{\mathbb{Z}}^{\mathscr{F}^2}$ there exist a $(0,0,[0))$-endomorphism $\mathfrak{e}_0$ of $\boldsymbol{B}_{\mathbb{Z}}^{\mathscr{F}^2}$ and an automorphism $\mathfrak{a}$ of $\boldsymbol{B}_{\mathbb{Z}}^{\mathscr{F}^2}$ such that $\mathfrak{e}=\mathfrak{e}_0\mathfrak{a}$. Moreover, $\mathfrak{e}=\mathfrak{e}_0\widetilde{\mathfrak{a}}^{2s}=\mathfrak{e}_0\mathfrak{h}_s$ in the case when $(0,0,[0))\mathfrak{e}=(s,s,[0))$, and $\mathfrak{e}=\mathfrak{e}_0\widetilde{\mathfrak{a}}^{2s+1}=\mathfrak{e}_0\mathfrak{h}_{s}\widetilde{\mathfrak{a}}$ in the case when $(0,0,[0))\mathfrak{e}=(s,s,[1))$ for some integer $s$.
\end{theorem}

\begin{proof}
By Lemma~\ref{lemma-3.2} there exists an automorphism $\mathfrak{a}$ of the semigroup $\boldsymbol{B}_{\mathbb{Z}}^{\mathscr{F}^2}$ such that $(0,0,[0))\mathfrak{e}=(0,0,[0))\mathfrak{a}$. Then the product $\mathfrak{e}\mathfrak{a}^{-1}$ is a $(0,0,[0))$-endomorphism of $\boldsymbol{B}_{\mathbb{Z}}^{\mathscr{F}^2}$. Let be $\mathfrak{e}_0=\mathfrak{e}\mathfrak{a}^{-1}$. Since for an arbitrary monoid $S$ every right translation $\rho_u\colon S\to S$, $s\mapsto su$ on an element of the group of units of $S$ is a bijective map, we conclude that  the equality $\mathfrak{e}_0=\mathfrak{e}\mathfrak{a}^{-1}$ implies that $\mathfrak{e}=\mathfrak{e}_0\mathfrak{a}$. The last statement follows from the second statemnet of Lemma~\ref{lemma-3.2}.
\end{proof}

Since the composition of two injective maps is an injective map, Theorems~\ref{theorem-2.6} and~\ref{theorem-3.3} imply the following theorem, which describes the structure of all injective endomorphisms of the semigroup $\boldsymbol{B}_{\mathbb{Z}}^{\mathscr{F}^2}$.

\begin{theorem}\label{theorem-3.4}
For any injective endomorphism $\mathfrak{e}$ of the semigroup $\boldsymbol{B}_{\mathbb{Z}}^{\mathscr{F}^2}$ there exist an injective  $(0,0,[0))$-endomorphism $\mathfrak{e}_0$ of $\boldsymbol{B}_{\mathbb{Z}}^{\mathscr{F}^2}$ and an automorphism $\mathfrak{a}$ of $\boldsymbol{B}_{\mathbb{Z}}^{\mathscr{F}^2}$ such that $\mathfrak{e}=\mathfrak{e}_0\mathfrak{a}$. Moreover, $\mathfrak{e}=\mathfrak{e}_0\widetilde{\mathfrak{a}}^{2s}=\mathfrak{e}_0\mathfrak{h}_s$ in the case when $(0,0,[0))\mathfrak{e}=(s,s,[0))$, $\mathfrak{e}=\mathfrak{e}_0\widetilde{\mathfrak{a}}^{2s+1}=\mathfrak{e}_0\mathfrak{h}_{s}\widetilde{\mathfrak{a}}$ in the case when $(0,0,[0))\mathfrak{e}=(s,s,[1))$ for some integer $s$, and one of the following conditions holds:
\begin{enumerate}
  \item\label{theorem-2.3(1)} there exist a positive integer $k$ and $p\in\{0,\ldots,k-1\}$ such that $\mathfrak{e}_0=\alpha_{k,p}$;
  \item\label{theorem-2.3(2)} there exist a positive integer $k\geqslant 2$ and $p\in\{1,\ldots,k-1\}$ such that $\mathfrak{e}_0=\beta_{k,p}$.
\end{enumerate}
\end{theorem}



\end{document}